\documentclass[12pt]{amsart}
\usepackage{amsmath,amsthm}
\usepackage{amssymb}

\usepackage{enumitem}
\setlist{label={\roman{enumi}\kern1pt$)$}}

\usepackage{xcolor}
\usepackage{graphicx}

\usepackage[T1]{fontenc}
\newtheorem{theorem}{Theorem}[section]
\newtheorem{corollary}[theorem]{Corollary}
\newtheorem{lemma}[theorem]{Lemma}
\newtheorem{proposition}[theorem]{Proposition}

\theoremstyle{definition}
\newtheorem{definition}[theorem]{Definition}
\newtheorem{remark}[theorem]{Remark}
\newtheorem{example}[theorem]{Example}

\numberwithin{equation}{section}

\newcommand{\St}{\mathcal{S}}

\newcommand{\M}{\mathcal{M}}
\newcommand{\HH}{\mathcal{H}}
\newcommand{\KK}{\mathcal{K}}

\newcommand{\N}{\mathcal{N}}
\newcommand{\X}{\mathcal{X}}
\newcommand{\Y}{\mathcal{Y}}
\newcommand{\Z}{\mathcal{Z}}
\newcommand{\mc}[1]{\mathcal{#1}}
\newcommand{\T}{\mathcal{T}}
\newcommand{\ol}{\overline}
\newcommand{\clran}{\ol{\mathrm{ran}}\,}

\newcommand{\cldom}{\ol{\mathrm{dom}}\,}

\newcommand{\PI}[2]{\left\langle \,#1 , #2\, \right\rangle}

\newcommand{\PMN}{P_{\M,\N}}
\newcommand{\PXY}{P_{\X,\Y}}
\DeclareMathOperator{\ran}{ran}
\DeclareMathOperator{\dom}{dom}

\DeclareMathOperator{\mul}{mul}
\DeclareMathOperator{\Id}{Id}
\DeclareMathOperator{\Sp}{Sp}

\newcommand{\lr}{\mathrm{lr}}

\begin{document}

%%%%% To ease editing, for IMPAN journals add:

\baselineskip=17pt

%%%%%%%%%%%%%%%%

\title{Idempotent linear relations}

\author[M. L. Arias]{M. Laura Arias}
\address{Instituto Argentino de Matemática Alberto P. Calderón- CONICET\\
Buenos Aires, Argentina \\ and\\ Depto. de Matemática, Facultad de Ingeniería, Universidad de Buenos Aires\\
Buenos Aires, Argentina}
\email{lauraarias@conicet.gov.ar}

\author[M. Contino]{Maximiliano Contino}
\address{Instituto Argentino de Matemática Alberto P. Calderón- CONICET\\
	Buenos Aires, Argentina \\ and \\ Depto. de Matemática, Facultad de Ingeniería, Universidad de Buenos Aires\\
	Buenos Aires, Argentina}
\email{mcontino@fi.uba.ar }

\author[A. Maestripieri]{Alejandra Maestripieri}
\address{Instituto Argentino de Matemática Alberto P. Calderón- CONICET\\
	Buenos Aires, Argentina \\ and \\ Depto. de Matemática, Facultad de Ingeniería, Universidad de Buenos Aires\\
	Buenos Aires, Argentina}
\email{amaestri@fi.uba.ar}

\author[S. Marcantognini]{Stefania Marcantognini}
\address{Instituto Argentino de Matemática Alberto P. Calderón- CONICET\\
	Buenos Aires, Argentina\\ and \\
	Universidad Nacional de General Sarmiento – Instituto de Ciencias\\
	Los Polvorines, Pcia. de Buenos Aires, Argentina}
\email{smarcantognini@ungs.edu.ar }

\begin{abstract}
	A linear relation $E$ acting on a Hilbert space is idempotent if $E^2=E.$ A triplet of subspaces is needed to characterize a given idempotent: $(\ran E, \ran(I-E), \dom E),$ or equivalently,  $(\ker(I-E), \ker E, \mul E).$ The relations satisfying the inclusions $E^2 \subseteq E$ (sub-idempotent) or $E \subseteq E^2$ (super-idempotent) play an important role.  Lastly,  the adjoint and the closure of an idempotent linear relation are studied.
\end{abstract}

\maketitle

%% Classification and key words; note that the 2010 classification is used:

\renewcommand{\thefootnote}{}



\renewcommand{\thefootnote}{\arabic{footnote}}
\setcounter{footnote}{0}

%%%%%%%%

\section{Introduction}

The introduction of linear relations by von Neumann \cite{vN} was motivated by the need to define the adjoint of a non-densely defined operator
and in considering the inverses of certain operators.
%Some other applications of the linear relations include the development of the fixed point theory and the study of differential inclusions of evolution (see, for instance, \cite{11} and \cite{13}).
Semi-projections, which form the linear relation counterpart of the class of projection operators, appear when solving least-squares problems of linear relations (see \cite{Na2}).  
Linear relations provide the appropriate framework when dealing with control problems subject to generalized or nonstandard boundary conditions. In particular, they naturally occur if the normal
equations, which are used to characterize solutions
of various standard constrained or unconstrained least-squares problems, involve the adjoint
of a {non-densely defined }linear operator (cf. \cite {leenashed}). 

 A  linear operator $E$ is said to be a {\it projection} if $E^2 = E$, that is, if $\dom E$ (the domain of $E$) is $E-$invariant and $E^2x = Ex$ for all $x\in \dom E$. For any given projection $E$, if $\M := \ran E$ (the range of $E$) and $\N := \ker E$ (the kernel of $E$) then
 $${\rm (1)}~\M \subseteq \dom E, \mbox{\rm and  
 (2)}~\M\cap\N=\{0\}.$$
Unbounded (even non closable) projections were first considered by \^Ota \cite{Ota}. He showed that any projection $E$ is fully determined by its range and kernel, and that the projection determined by $(\M,\N)$ is closed if and only if both $\M$ and $\N$ are closed. Further investigations on closed densely defined projections were carried out by And\^o \cite{Ando}. We extended this work to \emph{semiclosed} projections in \cite{Semi}. 

%A subspace of a Hilbert space $\HH$ is semiclosed if it is the range of a closed operator \cite[Theorem 1.1]{Filmore} and a linear operator acting in $\HH$ is {\it semiclosed} when its graph is a semiclosed subspace of $\HH\times \HH$. The analog  of \^Ota's result reads that the projection $E$ is semiclosed if and only if both 
% $\M$ and $\N$ are semiclosed \cite[Proposition 3.2]{Semi}. 

Cross and Wilcox \cite{Cross} and Labrousse \cite{Labrousse} studied the linear relations satisfying (1) and $E^2=E$. Such linear relations are called {\it semi-projections} \cite{Labrousse} or {\it multivalued linear projections} \cite{Cross}.  As with projections, semi-projections  are fully characterized by the range and kernel, and  in this case the multivalued part is given by  their intersection. So a semi-projection is a projection if and only if (2) holds. Dropping not just (2), but both (1) and (2), the result is an {\it idempotent} relation; that is, a linear relation $E$ such that $E^2=E.$

Any idempotent $E$ verifies the twofold inclusion $E^2\subseteq E \subseteq E^2$. When a relation $E$ only satisfies the left inclusion, it is termed  {\it sub-idempotent}. Similarly,  $E$ is  {\it super-idempotent} if instead  the other inclusion holds.

The purpose of this paper is to study idempotents, as well as sub- and super-idempotents. Various characterizations of these classes are given, as well as adjoints and closures of relations in these classes. Much was already done for the class of semi-projections by Cross and Wilcox \cite{Cross} (see also \cite{Labrousse}).

Section 2 serves to introduce the notation and to give some preliminary results. In Section 3 we show that for a full description, three subspaces are needed; $(\ran E, \ran(I-E), \dom E)$ for sub-idempotents and  $(\ker(I-E), \ker E, \mul E)$ for super-idempotents.
 Then we turn our attention to the description of $E^2$ when $E$ is either sub- or super-idempotent, and we establish in either case that $E^2$ is an idempotent. In Section 4 the results of Section 3 are applied to obtain several characterizations of idempotents. The main results of this section concern the representation of the class of idempotents. These include two in which a triplet of subspaces uniquely determines an idempotent whenever the so-called {\it idempotency condition} is satisfied. Section 5 looks at the closure and adjoint of a relation $E$ which is one of the three classes. In general these operations do not yield idempotents. Necessary and sufficient conditions are given for $E^*$ and $\overline{E}$ to be idempotent, and we characterize those idempotents that are closed. Throughout, examples are presented illustrating the very rich structure of all these classes.

\section{Preliminaries}
\label{sec:preliminaries}

Throughout, $\HH,$ $\KK$ and  $\mc E$  are complex and separable Hilbert spaces.  As usual, the direct sum of two subspaces $\M$ and $\N$ of a Hilbert space $\HH$ is indicated by $\M \dotplus \N.$ 
The orthogonal complement of a subspace $\mc{M} \subseteq \HH$ is written as $\mc{M}^\perp,$ or $\HH \ominus \M$ interchangeably. 

We  consider the  inner product on $\HH \times \KK$ 
$$\PI{(h,k)}{(h',k')}=\PI{h}{h'}+\PI{k}{k'}, \ (h,k), (h',k') \in \HH \times \KK,$$
with the associated norm  $\Vert (h,k)\Vert^2=\Vert h \Vert^2+\Vert k \Vert^2.$

For $\St$ and $\mc T$  closed subspaces of $\HH,$ Friedrichs  \cite{Friedrichs} defined the cosine of the \emph{angle} between $\St$ and $\mc T$ as 
\begin{align*}
c(\St,\mc T):=\sup \left\{\vert \PI{x}{y}\vert \! \!: x\in \St\ominus(\St \cap \mc T), \! y\in\mc T \ominus(\St \cap \mc T),  \|x\|\ \!, \! \|y\|\leq 1 \right\}.
\end{align*}
On the other hand, the \emph{minimal angle} between $\St$ and $\mc T$ was defined by Dixmier \cite{Dixmier} as the one whose cosine is
$$c_0(\St,\mc T):=\sup \left\{|\PI{x}{y}|\,:\, x\in \St, \; y\in\mc T, \|x\|\ \!, \! \|y\|\leq 1 \right\}.$$

In general, $c(\St,\mc T)\leq c_0(\St,\mc T)$. However, when $\St\cap \mc T=\{0\}$ both angles coincide.

\begin{theorem}[{\cite[Theorem 13]{Deutsch}}] \label{prop angulos} Let $\St,\ \mc T$ be closed subspaces of $\HH.$ The following are equivalent:
\begin{enumerate}
\item $c(\St,\mc T)<1;$
\item $\St + \mc T$ is closed;
\item $\St^{\perp}+\T^{\perp}$ is closed.
\end{enumerate}
\end{theorem}

\begin{lemma} \label{lemaangulo}  Let $\St,\ \mc T, \mc  W$ be closed subspaces of $\HH$ such that $\mc T \subseteq \mc W$ and $\mc T \cap \St=\mc W \cap \St.$  Then 
$$c(\mc T, \St)\leq c(\mc W, \St).$$
\end{lemma}

\begin{proposition}[{\cite[Proposition 2.3.3, Corollary 2.3.1]{Labrousse1980}}] \label{rangeop} Let $\M, \N$ be operator ranges such that $\M+\N$ is closed. Then
	\begin{enumerate}
		\item[1.] $\ol{\M\cap \N}=\ol{\M}\cap \ol{\N}.$
		\item[2.] $(\M\cap \N)^{\perp}=\M^{\perp}+\N^{\perp}.$
	\end{enumerate}
\end{proposition}
\subsection*{Linear relations}

A linear relation from  $\HH$ into $\KK$ is a linear subspace $T$ of the cartesian product $\HH \times \KK.$ The set of linear relations from $\HH$ into $\KK$ will be denoted by $\lr(\HH,\KK),$ and $\lr(\HH):=\lr(\HH,\HH).$ 
The domain,  range, kernel or nullspace and multivalued part of $T\in \lr(\HH,\KK)$ are denoted by $\dom T,$  $\ran T,$ $\ker T$ and  $\mul T,$ respectively. When $\mul T=\{0\},$ $T$ is an operator.

\begin{lemma}[{\cite[Proposition 1.21]{Labrousse}}] \label{lemalr} Let $S, T \in \lr(\HH,\KK).$ Then $S=T$ if and only if $S\subseteq T,$ $\dom T\subseteq\dom S$ and $\mul T \subseteq \mul S.$
\end{lemma}

Given  $T, S \in \lr(\HH,\KK),$
$T \cap \St$  and $T \ \hat{+}  \ S$ are the usual intersection and sum of $T$ and $S$ as subspaces, respectively.
In particular, $\mul(T \ \cap  \ S)=\mul T\cap \mul S$ and $\ker(T \ \cap \ S)=\ker T \cap \ker S,$  $\dom(T \ \hat{+}  \ S)=\dom T+\dom S$ and $\ran(T \ \hat{+}  \ S)=\ran T+\ran S.$ 

\medskip
The sum of two linear relations $T, S \in \lr(\HH,\KK)$ is the linear relation defined by
$$T+S:=\{(x,y+z): (x,y ) \in T \mbox{ and } (x,z) \in S\}.$$

If $T \in \lr(\HH,\mc E)$ and $S\in \lr(\mc E,\mc K),$ the product $ST$ is the linear relation from $\HH$ to $\KK$ defined by
$$ST:=\{(x,y): (x,z) \in T \mbox{ and } (z,y) \in S \mbox{ for some } z \in \mc{E}\}.$$

Given a subspace $\M$ of $\HH$, $I_\M:=\{(u,u): u\in\M\}.$ In particular, the identity is $I:=I_\HH.$

\begin{lemma} \label{resta} \label{resta2} Let $T \in \lr(\HH).$ Then $(u,v) \in I-T$ if and only if $(u,u-v) \in T.$ As a consequence, $\ker(I-T) \subseteq \ran T \cap \dom T,$ $\ker(I-T) \subseteq \ker(I-T^2)$ and $\ran(I-T^2) \subseteq \ran(I-T).$
\end{lemma}

The inverse of $T\in\lr(\HH,\KK)$ is $T^{-1}=\{(y,x): (x,y)\in T\}.$  The following identities can be easily checked
\begin{equation} \label{TT}
T^{-1}T=I_{\dom T} \ \hat{+} (\{0\} \times  \ker T) \ \mbox{ and } \ TT^{-1}=I_{\ran T} \ \hat{+} (\{0\} \times  \mul T),
\end{equation}
 \cite[Equation 2.4]{Hassi}.

The closure $\ol{T}$ of a linear relation $T$ from $\HH$ to $\KK$ is the closure of the  subspace $T$ in $\HH \times \KK,$ when the product is provided with the product topology. 
The relation $T$ is  \emph{closed} when it is closed as a subspace of $\HH \times \KK.$ 

The \emph{adjoint} of $T$ is the linear relation from $\KK$ to $\HH$ defined by $$T^*:=JT^{\perp}=(JT)^{\perp},$$ where $J(x,y)=i(-y,x).$ The adjoint 
is automatically a closed linear relation and $\ol{T}=T^{**}:=(T^*)^*.$  It is immediate that $(\ol{T})^*=T^*.$  Since
$$T^*=\{(x,y) \in \KK \times \HH: \PI{g}{x}=\PI{f}{y} \mbox{ for all } (f,g) \in T \},$$ we get that $\mul T^* =(\dom T)^{\perp}$ and $\ker T^* =(\ran T)^{\perp}.$ Therefore, if $T$ is closed both $\ker T$ and $\mul T$ are closed subspaces.

\begin{theorem}[{\cite[Theorem 3.3]{Cross}}] \label{closedrange} Let $T \in \lr(\HH,\KK)$ be closed. Then $\ran T $ is closed if and only if $\ran T^*$ is closed.
\end{theorem}

If $T \in \lr(\HH,\mc E)$ and $S\in \lr(\mc E,\mc K)$  then
\begin{equation} \label{product}
	T^*S^* \subseteq (ST)^*.	
\end{equation}
 If $T, S \in \lr(\HH,\KK)$ then 
\begin{equation} \label{sumadj}
	(T \ \hat{+} \ S )^* = T^* \cap S^*, 
\end{equation} and 
\begin{equation} \label{sum}
	T^*+S^* \subseteq (T+S)^*.	
\end{equation}

\begin{lemma}[{\cite[Lemma 2.10]{Hassi2}}] \label{sumclosed} Let $T, S \in \lr(\HH, \KK).$ Then $T \ \hat{+} \ S$ is closed if and only if $T^* \ \hat{+} \ S^*$ is closed.
\end{lemma}

%The next two results offer some new characterizations of the adjoint of a linear relation.

\begin{theorem}[{\cite[Theorem 4.2]{Sandovici}}]  \label{sando1} Let $A, B \in \lr(\HH, \KK)$ such that $A \subseteq B^*.$ If
\begin{equation} \label{iden}
\ker A+ \ran B=\HH  \ \mbox{ and } \ \ker B+\ran A=\KK,
\end{equation}
then $A=B^*$ and $B=A^*$ and both $A$ and $B$ are closed with closed ranges.
\end{theorem}

%\begin{corollary}[{\cite[Corollary 4.3]{Sandovici}}]  \label{sando2} Let $A, B \in \lr(\HH, \KK)$ such that $A \subseteq B^*$ and the following hold:
%\begin{equation} \label{iden}
%\mul A+ \dom B=\KK  \ \mbox{ and } \ \mul B+\dom A=\HH.
%\end{equation}
%Then, $A=B^*$ and $B=A^*,$ and both $A$ and $B$ are closed with closed domains, respectively.
%\end{corollary}

\section{Sub- and super-idempotents}
A linear relation $E \subseteq \HH \times \HH$ is called an \emph{idempotent} %or a \emph{projection} 
if $E^2=E.$ If, in addition $\ran E \subseteq \dom E$, we say that $E$ is a \emph{semi-projection}. If $E$ is an idempotent operator then $\ran E \subseteq \dom E$ and we say that $E$ is a \emph{projection}.
Denote by $\Id(\HH)$ and $\Sp(\HH)$ the set of idempotents and the set of semi-projections, respectively.

Semi-projections  are studied in detail in \cite{Cross} and \cite{Labrousse} where,  among other results, it is proved that a semi-projection is uniquely determined by its range and kernel. More precisely, 
if $\M$ and $\N$ are two subspaces of $\HH$  then
\begin{equation}\label{semiproy-rep}
P_{\M, \N}:=I_\M \ \hat{+} \ (\N \times \{0\})
\end{equation}
is the unique semi-projection with
 $\ran P_{\M, \N}=\M$ and $\ker P_{\M, \N}=\N.$ Furthermore, $\dom P_{\M, \N} =\M + \N$ and  $\mul P_{\M, \N}=\M \cap \N.$ 
 
 Semi-projections appear, for example, when solving least-squares problems for linear relations (see \cite[Proposition 2.2]{Na2}). More precisely, if $T \in \lr(\HH, \KK)$ then, by \eqref{TT}, $T^{-1}T=P_{\dom T, \ker T}$ and  $TT^{-1}=P_{\ran T, \mul T}.$

\begin{proposition}[{\cite[Proposition 1.1]{Cross}}] \label{L1} Let $E \in \lr(\HH).$ Then $E \in \Sp(\HH)$ if and only if $E=P_{\ran E, \ker E}.$
\end{proposition}

One of our goals is to get a representation similar to (\ref{semiproy-rep}) for idempotent relations. The range and kernel  are not sufficient to fully describe an idempotent unless it is a semi-projection. We will see that a triplet of subspaces is needed to characterize an idempotent relation.  

\begin{example} If $\M, \St, \St'$ are subspaces of $\HH$ with $\M\dot{+}\St=\M\dot{+}\St'$ and $\St \not =\St',$ it can be seen that the relations $E=I_\M\hat{+} (\{0\}\times \St)$ and $E'=I_\M\hat{+} (\{0\}\times \St')$ are idempotent with $\ran E=\ran E'=\M\dotplus\St$ and $\ker E=\ker E'=\{0\}$ although $E\neq E'$ because $\mul E=\St$ and $\mul E'=\St'.$
\end{example}

\medskip
Given a linear relation $E,$ there are two semi-projections naturally associated with $E,$ namely $P_{\ker (I-E), \ker E}$ and $P_{\ran E, \ran(I-E)},$ as the following lemma shows.

\begin{lemma} \label{always} Let $E \in \lr(\HH).$ Then 
	$$P_{\ker (I-E), \ker E} \ \hat{+} \ (\{0\} \times \mul E) \subseteq E \subseteq P_{\ran E, \ran(I-E)} \cap\ (\dom E \times \HH).$$
\end{lemma}
\begin{proof} To see the first inclusion we only need to check that $I_{\ker(I-E)} \subseteq E$ but if $u \in \ker(I-E)$ then $(u,0) \in I-E$ whence $(u,u) \in E.$ To prove the second inclusion, let $(u,v) \in E$ then $(u,u-v) \in I-E$ so that $(u,v)=(v,v)+(u-v,0) \in P_{\ran E, \ran(I-E)}.$
\end{proof}

From now on $\M, \N$ and $\St$ are subspaces of $\HH.$

In view of the above lemma, we begin by studying the relations
\begin{equation} \label{RT}
R:=\PMN \cap (\St\times \HH) \mbox{ and } T:=\PMN \ \hat{+}  \ (\{0\} \times \St).
\end{equation}

\begin{lemma} \label{propmul}  \label{cormul2} Let $T,R\in \lr(\HH)$ be defined as in (\ref{RT}). 
Then
\begin{enumerate}
\item [1.] $\dom R= (\M + \N) \cap \St,$ $\ran R=\M \cap (\N+ \St),$  $\ker R=\N\cap \St$ and $\mul R=\M \cap \N.$ 
\item [2.] $\dom T=\M+\N,$ $\ran T=\M+\St,$  $\ker T=\N+ \M \cap \St$ and $\mul T= \St + \M \cap \N.$
\end{enumerate}
\end{lemma}

\begin{proof}  Use the definitions of $R$ and $T.$
\end{proof}

\begin{lemma} \label{propmul2}  Let $R, T \in \lr(\HH)$ be defined as in (\ref{RT}). 
	Then
	\begin{enumerate}
		\item [1.] $I-R=P_{\N,\M} \cap (\St\times \HH)$ and $R^{-1}=P_{\St,\N} \cap (\M \times \HH).$ 
		\item [2.] $I-T=P_{\N,\M} \ \hat{+}  \ (\{0\} \times \St)$ and $T^{-1}=P_{\M,\St} \ \hat{+}  \ (\{0\} \times \N).$
	\end{enumerate}
\end{lemma}

\begin{proof}  By Lemma \ref{resta}, $(x,y) \in I-R$ if and only if $(x,x-y) \in R,$ or equivalently $(x,x-y) \in P_{\M,\N} \cap (\St\times \HH).$ Then $(x,y) \in I-R$ if and only if $x=m+n \in \St,$ $m\in \M$, $n \in \N$ and $x-y =m.$ Therefore $y=x-m=n$ so that $(x,y)=(m+n,n),$ $m+n \in \St.$ Hence $(x,y) \in P_{\N,\M} \cap (\St\times \HH).$ The other inclusion is similar. 

To prove the formula for $R^{-1},$ let $(x,y) \in R.$ Then $(x,y)=(m+n,m)$ where $m \in \M, n \in \N$ and $m+n=:s \in \St.$ Therefore $(x,y)=(s,s-n)$ with $s-n=m \in \M$ so that $(y,x)=(s-n,s) \in P_{\St,\N} \cap (\M \times \HH).$ Hence $R^{-1} \subseteq P_{\St,\N} \cap (\M \times \HH).$ The reverse inclusion is similar. 

The proof of item $2$ follows in a similar fashion.
\end{proof}

%\begin{lemma} \label{propmul3}  Let $\M', \N', \St'$ be subspaces of $\HH.$ Then
%\begin{enumerate}
%\item [1.]  $P_{\M,\N} \cap (\St \times \HH)  \cap P_{\M',\N'} \cap (\St' \times \HH) =P_{\M\cap \M',\N \cap \N'} \cap ((\St \cap \St') \times \HH).$
%\item [2.] $P_{\M, \N}  \hat{+}  \ (\{0\} \times \St)  \ \hat{+} \ P_{\M', \N'}  \hat{+}  \ (\{0\} \times \St')=P_{\M+\M', \N+\N'}  \hat{+}  \ (\{0\} \times (\St+\St')).$
%\end{enumerate}
%\end{lemma}

\begin{lemma} \label{always2} Let  $R,T \in \lr(\HH)$ be defined as in (\ref{RT}). 
Then
\begin{enumerate}
\item [1.] $R^2 \subseteq R.$ 
\item [2.] $T \subseteq T^2.$
\end{enumerate}
\end{lemma}

\begin{proof} $1:$ By Lemmas \ref{propmul} and \ref{propmul2}, it easily follows that $\ker(I-R)=\ran R \cap \dom R.$  If $(x,y) \in R^2$ then there exists $z \in \HH$ such that $(x,z), (z,y) \in R.$ So that $z \in \ker (I-R),$ or equivalently $(z,z) \in R.$ Hence $(x-z,0)=(x,z)- (z,z) \in R$ and $(0,y-z)=(z,y)-(z,z) \in R.$ Therefore $(x,y)=(x-z,0)+(z,z)+(0,y-z) \in R.$
	
$2$: By Lemmas \ref{propmul} and \ref{propmul2}, it easily follows that $\ran(I-T)=\ker T +\mul T.$ If $(x,y) \in T$ then $(x,x-y) \in I-T$ so that $x-y \in \ker T + \mul T.$ Hence $x-y=n+s$ for some $n \in \ker T$ and $s  \in \mul T.$ Then $(x,y+s)=(x,y)+(0,s) \in T$ and $(y+s,y)=(x-n,y)=(x,y)-(n,0) \in T.$ Therefore $(x,y) \in T^2.$
\end{proof}

\begin{definition} $E \in \lr(\HH)$ is called \emph{sub-idempotent} if $E^2 \subseteq E$ and \emph{super-idempotent} if $E \subseteq E^2.$ 
\end{definition}

\begin{lemma} If  $E \in \lr(\HH)$ is sub- (super-idempotent) then $E^2$ is  sub- (super-idempotent).
\end{lemma}
\begin{proof} Use that if $A, B \in \lr(\HH)$ and $A \subseteq B$ then $A^2 \subseteq B^2.$ 
\end{proof}

\begin{lemma} \label{subsupker} Let $E \in \lr(\HH).$ 
\begin{enumerate}
\item[1.] If $E$ is sub-idempotent then $\ker E^2=\ker E,$ $\ker(I-E^2)=\ker(I-E)$ and $\mul E^2=\mul E.$
\item[2.]  If $E$ is super-idempotent then $\ran E^2=\ran E,$ $\ran(I-E^2)=\ran(I-E)$ and $\dom E^2=\dom E.$
\end{enumerate}
\end{lemma}

\begin{proof} By Lemma \ref{resta}, $\ker(I-E)  \subseteq \ker(I-E^2)$ always holds. If $E$ is sub-idempotent then $I-E^2 \subseteq I-E$ and then $\ker (I-E^2) \subseteq \ker(I-E).$ The other assertions follow similarly.
\end{proof}

\begin{proposition} \label{sub}\label{corsub} Let $E \in \lr(\HH).$ Then the following are equivalent:
\begin{enumerate}
\item $E$ is sub-idempotent;
\item $E= P_{\ran E, \ran(I-E)} \cap (\dom E \times \HH);$
\item $\ker (I-E)=\ran E \cap \dom E;$
\item $P_{\ran E \cap \dom E, \ker E}  \subseteq E;$
\end{enumerate}
In this case, 
$\mul E \cap \dom E=\ran E \cap \ker E.$
\end{proposition}
\begin{proof}  Set $R:=P_{\ran E, \ran(I-E)} \cap (\dom E \times \HH).$ By Lemma \ref{always}, $E \subseteq R$ and $\dom E =\dom R$.  

$i) \Rightarrow ii)$: Suppose that $E^2 \subseteq E.$ To see that $E=R$ we apply Lemma \ref{lemalr} by showing that $\mul R \subseteq \mul E.$ Let $w \in \mul R =\ran E \cap \ran(I-E).$ Then there exist $u, v \in \HH$ such that $(u,w) \in E$ and $(v,w) \in I-E.$ Then $(v,v-w) \in E,$ so that $(u+v,v) \in E.$ Hence $(u+v,v-w) \in E^2 \subseteq E.$ Therefore $(0,w)=(u+v,v)-(u+v,v-w) \in E.$

$ii) \Rightarrow iii)$: Follows from  Lemma \ref{propmul}. 

$iii) \Rightarrow iv)$: $P_{\ran E \cap \dom E, \ker E}=P_{\ker(I-E),\ker E } \subseteq E,$ by Lemma \ref{always}.
 
$iv) \Rightarrow i)$: Let $(u,v) \in E^2.$ Then there exists $w$ such that $(u,w),$ $(w,v) \in E.$ Then $(w,w) \in E$ because $w \in \ran E \cap \dom E$ and $I_{\ran E \cap \dom E} \subseteq E.$ Hence $(u,v)=(u-w,0)+(w,w)+(0,v-w) \in E.$

In this case, $\mul E=\ran E \cap \ran(I-E).$ By Lemma \ref{propmul2}, $I-E$ is also sub-idempotent; then $\ker E=\ran (I-E)\cap \dom E.$ Therefore $\ran E \cap \ker E=\ran E \cap \ran(I-E) \cap \dom E= \mul E \cap \dom E.$
\end{proof}

\begin{corollary} \label{subcor1} The set of sub-idempotents is $$\{P_{\M,\N} \cap (\St \times \HH)\}.$$
\end{corollary}
\begin{proof} By Proposition \ref{sub}, any sub-idempotent belongs to the set. Conversely, if $R:=\PMN \cap (\St\times \HH),$ by Lemma \ref{always2}, $R$ is sub-idempotent. 
\end{proof}

\begin{remark} Let $E \in \lr(\HH)$ be sub-idempotent. Then $E=P_{\M,\N} \cap  (\St \times \HH)$ if and only if $\ran E=(\N+\St) \cap \M,$ $\ran(I-E)=(\M+\St) \cap \N$ and $\dom E=(\M + \N) \cap \St.$ 
	
In fact, since  $E$ is sub-idempotent, by Proposition \ref{sub}, $$E=P_{\ran E, \ran(I-E)} \cap (\dom E \times \HH).$$ So that, if $\ran E=(\N+\St) \cap \M,$ $\ran(I-E)=(\M+\St) \cap \N$ and $\dom E=(\M + \N) \cap \St,$  then  $E=P_{(\N+\St) \cap \M, (\M+\St) \cap \N} \hat{+} (((\M +\N) \cap \St) \times \HH)=P_{\M,\N} \cap (\St \times \HH).$ The converse follows from Lemma \ref{propmul}.
\end{remark}

\begin{proposition} \label{supra} Let $E \in \lr(\HH).$ Then the following are equivalent:
	\begin{enumerate}
		\item $E$ is super-idempotent;
		\item $E=P_{\ker(I-E), \ker E} \ \hat{+} \ (\{0\}\times \mul E);$
		\item $\ran(I-E)=\ker E + \mul E;$
		\item $E\subseteq P_{\ran E,\ker E+\mul E}.$
	\end{enumerate}	
In this case, $\dom E = \ran E \cap \dom E+\ker E.$
\end{proposition}

\begin{proof}  $i) \Rightarrow ii)$:  Suppose that $E \subseteq E^2$ and let $(u,v) \in E,$ so that there exists $w$ such that $(u,w), (w,v) \in E.$ Then $(u-w,0) \in E,$ $(0,v-w) \in E$ and $(w,w)=(u,v)-(u-w,0)-(0,v-w) \in E.$ Therefore $(w,0) \in I-E$  and $(u,v)=(w,w)+(u-w, v-w) \in P_{\ker(I-E), \ker E} \ \hat{+} \ (\{0\}\times \mul E).$ This shows that $E \subseteq  P_{\ker(I-E), \ker E} \ \hat{+} \ (\{0\}\times \mul E).$ The other inclusion always holds (see Lemma \ref{always}).

$ii) \Rightarrow iii)$: Follows from Lemma \ref{propmul}. 

$iii) \Rightarrow iv)$: By Lemma \ref{always}, $E\subseteq P_{\ran E,\ran(I-E)}=P_{\ran E, \ker E +\mul E}.$

$iv) \Rightarrow i)$: Let $(u,v) \in E.$ Then $(u,v)=(x+y,x),$ with $ x\in \ran E,$ $y \in \ker E+ \mul E.$ So $y=y_1+y_2$ with $(y_1,0), (0,y_2)\in E.$ Then $(x+y,x+y_2),(x+y_2,x) \in E$ so that $(x+y,x)=(u,v) \in E^2.$

In this case, $\dom E \subseteq \dom(P_{\ran E,\ker E+\mul E})=\ran E+\ker E.$ So that $\dom E \subseteq \ran E \cap \dom E+\ker E.$ The other inclusion always holds. 
\end{proof}

\begin{corollary} \label{supcor1} The set of super-idempotents is $$\{P_{\M,\N} \ \hat{+} \ (\{0\} \times \St)\}.$$
\end{corollary}

\begin{proof}  By Proposition \ref{supra}, any super-idempotent belongs to the set. Conversely, if $T:=\PMN   \ \hat{+} \ (\{0\} \times \St),$ by Lemma \ref{always2},  $T$ is super-idempotent. 
\end{proof}

\begin{remark} Let $E \in \lr(\HH)$ be super-idempotent. 

Then $E=P_{\M,\N} \ \hat{+} \ (\{0\} \times \St)$ if and only if $\ker(I-E)=\M+\N\cap \St,$ $\ker E=\N+\M \cap \St$ and $\mul E=\St + \M \cap \N.$ 
	
In fact, since   $E$ is super-idempotent, by Proposition \ref{supra}, $$E=P_{\ker(I-E), \ker E} \ \hat{+} \ (\{0\}\times \mul E).$$ So that, if $\ker(I-E)=\M+\N\cap \St,$ $\ker E=\N+\M \cap \St$ and $\mul E=\St + \M \cap \N,$ then  $E=P_{\M+\N \cap \St, \N+\M \cap \St} \hat{+} (\{0\} \times (\St+\M\cap\N))=P_{\M,\N} \ \hat{+} \ (\{0\} \times \St).$ The converse follows from Lemma \ref{propmul}.
\end{remark}

\begin{corollary} \label{subsupra} Let $E \in \lr(\HH).$ Then $E$ is sub-idempotent if and only if $I-E$  is sub-idempotent if and only if  $E^{-1}$  is sub-idempotent. An analogue result holds if $E$ is super-idempotent.
\end{corollary}
\begin{proof} Use Lemma \ref{propmul2} and Corollaries \ref{subcor1} and \ref{supcor1}.
\end{proof}

\begin{proposition} \label{lemaIC} \label{propsub} \label{corsupra} The following statements hold:
\begin{enumerate}
\item[1.] Let $E$ be sub-idempotent. Then $E \in \Id(\HH)$ if and only if $\dom E=\ran E \cap \dom E +\ker E.$
\item[2.]  Let $E$ be super-idempotent. Then $E \in \Id(\HH)$ if and only if $\mul E \cap \dom E =\ran E \cap \ker E.$
\end{enumerate}
\end{proposition}

\begin{proof} $1$:  If $E \in \Id(\HH)$ then, by Proposition \ref{supra}, $\dom E = \ran E \cap \dom E +\ker E.$
Conversely, since $E$ is sub-idempotent, by Proposition \ref{sub}, 
$$T:=P_{\ran E \cap \dom E, \ker E} \ \hat{+} \ (\{0\} \times \mul E)\subseteq E$$ and  $\mul E \subseteq \mul T.$  Since $\dom E=\ran E \cap \dom E +\ker E =\dom T,$ by Lemma \ref{lemalr}, $E=T.$ Then, by Proposition \ref{supra}, $E$ is super-idempotent, so that $E \in \Id(\HH).$

$2$: If $E \in \Id(\HH)$ then, by Proposition \ref{sub}, $\mul E \cap \dom E =\ran E \cap \ker E.$
Conversely, since $E$ is super-idempotent, by Proposition \ref{supra}, $E \subseteq P_{\ran E, \ran(I-E)}.$ By Lemma \ref{always}, $E\subseteq P_{\ran E, \ran(I-E)} \cap (\dom E \times \HH):=R$ and $\dom R=(\ran E+\ran(I-E))\cap \dom E=\dom E.$ Also, $\mul R= \ran E \cap \ran (I-E)=\ran E \cap (\ker E+ \mul E)=\ran E \cap \ker E+\mul E = \mul E.$ Then $E=R,$ and by Proposition \ref{supra}, $E$ is sub-idempotent. Therefore $E \in \Id(\HH).$
\end{proof}

\begin{theorem} \label{corsubsup} Let $E \in \lr(\HH)$. Then
\begin{enumerate}
\item[1.]  $E$ is sub-idempotent if and only if  $$E^2=P_{\ker(I-E), \ker E}  \ \hat{+} \ (\{0\} \times \mul E).$$
\item[2.]  $E$ is super-idempotent if and only if $$E^2=P_{\ran E, \ran(I-E)} \cap (\dom E \times \HH).$$
\end{enumerate}
In either case, $E^2 \in \Id(\HH).$
\end{theorem}

\begin{proof} $1$:   Set $P:=P_{\ker(I-E), \ker E}  \ \hat{+} \ (\{0\} \times \mul E).$
	
If $E^2=P$ then, since $P \subseteq E$ always holds, $E^2 \subseteq E$ and $E$ is sub-idempotent.  

Conversely, if $E^2 \subseteq E,$ by Lemma \ref{subsupker}, $P=P_{\ker(I-E^2), \ker E^2}  \ \hat{+} \ (\{0\} \times \mul E^2) \subseteq E^2.$ Also, $\mul E^2=\mul E \subseteq \mul P.$ It only remains to see that $\dom E^2\subseteq \dom P$ to apply Lemma \ref{lemalr} and get that $E^2=P.$ Let $x \in \dom E^2;$ then there exists $w \in\ran E \cap \dom E$ such that $(x,w), (w,y) \in E$ for some $y \in \HH.$ But, by Proposition \ref{sub}, $\ran E \cap \dom E=\ker(I-E);$ then $w \in \ker(I-E)$ or $(w,w) \in E.$ Hence $(x-w,0) \in E$ and $x=x-w+w\in \ker E+\ker(I-E)=\dom P.$ 

$2$: Set $Q:=P_{\ran E, \ran(I-E)} \cap (\dom E \times \HH).$

 If $E^2=Q,$ since $E \subseteq Q$ always holds, $E \subseteq E^2$ and $E$ is super-idempotent. 
 
 Conversely, if  $E \subseteq E^2,$  by Lemma \ref{subsupker}, $E^2 \subseteq Q=P_{\ran E^2, \ran(I-E^2)} \cap (\dom E^2 \times \HH).$   Also, $\dom Q= \dom E = \dom E^2.$ It only remains to see that $\mul Q \subseteq \mul E^2$ to apply Lemma \ref{lemalr} and get that $E^2=Q.$ 
By Lemma \ref{propmul} and Proposition \ref{supra}, $\mul Q=\ran E\cap\ran(I-E)=\ran E \cap(\ker E+\mul E)=\ran E\cap\ker E+\mul E\subseteq \ran E\cap\ker E+\mul E^2,$ where the inclusion holds because $E\subseteq E^2.$  To see that $\ran E \cap\ker E\subseteq \mul E^2,$  let $u\in \ran E \cap\ker E.$ Thus $(u,0)\in E$ and $(y,u)\in E$ for some $y\in\dom E .$ Hence $(y,0)\in E^2$ and $(y,u)\in E^2.$ Therefore $(0,u)\in E^2,$ i.e. $u\in \mul E^2.$

Finally, suppose that $E$ is sub-idempotent then $E^2$ is sub-idempotent. Since $E^2=P,$ it is also super-idempotent and then $E^2\in \Id(\HH).$ The case when $E$ is super-idempotent is similar.
\end{proof}

\begin{corollary} \label{propsubsuper2} Let $E \in \lr(\HH)$. Then:
\begin{enumerate}
\item[1.]  $E$ is sub-idempotent if and only if  $\ker E^2=\ker E,$ $\ker(I-E^2)=\ker(I-E),$ $\mul E^2=\mul E$ and $E^2$ is super-idempotent.
\item[2.]  $E$ is super-idempotent if and only if $\ran E^2=\ran E,$ $\ran(I-E^2)=\ran(I-E),$ $\dom E^2=\dom E$ and $E^2$ is sub-idempotent.
\end{enumerate}
\end{corollary}

\begin{proof} If $E$ is sub-idempotent then, by Lemma \ref{subsupker} and Theorem \ref{corsubsup}, the result follows. Conversely, by Proposition \ref{supra}, $$E^2=P_{\ker(I-E^2), \ker E^2}  \ \hat{+} \ (\{0\} \times \mul E^2)=P_{\ker(I-E), \ker E}  \ \hat{+} \ (\{0\} \times \mul E).$$ Then, by Theorem \ref{corsubsup}, $E$ is sub-idempotent. 
The case when $E$ is super-idempotent is similar. 
\end{proof}

The following example shows that there are linear relations which are sub-idempotent but not super-idempotent and viceversa.

\begin{example} \label{Ex1}
\begin{enumerate}
\item[1.] Take $R:=\PMN \cap (\St\times \HH),$ with  $\M\cap \St+ \N \cap \St \subsetneq (\M+\N) \cap \St.$ Then $R$ is sub-idempotent but not super-idempotent. In fact, by Lemma \ref{propmul},  $\dom R \cap \ran R+\ker R  = \M\cap \St+ \N \cap \St \subsetneq (\M+\N) \cap \St =\dom R.$ Then, by Proposition \ref{lemaIC}, $R$ is not super-idempotent.
	
\item[2.]  Take $T:=\PMN   \ \hat{+} \ (\{0\} \times \St),$ with $\St+\N\cap \M \subsetneq \St + \N \cap (\M+\St).$ Then $T$ is super-idempotent but not sub-idempotent. Indeed, by Lemma \ref{propmul}, $\mul T = \St+\N\cap \M \subsetneq \St+\N \cap (\M+\St)=\ran T \cap \ran(I-T).$ Therefore, by Proposition \ref{sub}, $T$ is not sub-idempotent.
\end{enumerate}
\end{example}

\begin{remark} If $E$ is a super-idempotent operator on $\HH,$ since $\mul E=\{0\},$ by Proposition \ref{supra}, $E=P_{\ker(I-E),\ker E}$ and $\ker(I-E) \cap \ker E=\{0\}.$ So that $E$ is a projection.

If $E$ is a sub-idempotent operator then  $E=P_{\ran E, \ran(I-E)} \cap (\dom E \times \HH)$ and $\{0\}=\mul E =\ran E \cap \ \ran(I-E).$ Then $E$ is a restriction of a projection and $E$ is a projection if and only if $\dom E=\ran E+ \ran(I-E).$  
\end{remark}

\section{Idempotent linear relations}
We begin this section with a list of corollaries regarding idempotent relations which follow immediately from the results in the previous section and the fact that $E\in\Id(\HH)$ if and only if $E$ is sub- and super-idempotent.

\begin{corollary}[{\cite[Propositions 2.2 and 2.4]{Labrousse}}] \label{propidemL} Let $E, F  \in \lr(\HH).$ Then 
$E\in \Id(\HH)$ if and only if  $I-E\in \Id(\HH)$ if and only if $E^{-1}\in \Id(\HH).$ 
%	\item [2.] If $E, F \in \Id(\HH)$  then $E \cap F \in \Id(\HH)$ and $E \ \hat{+} \ F \in \Id(\HH).$ 
%\end{enumerate}	
\end{corollary}
\begin{proof} Apply Corollary \ref{subsupra}.
\end{proof}

\begin{example}
Suppose that $E$ is a projection on $\HH.$ By Corollary \ref{propidemL}, $E^{-1}$ is idempotent but since $\dom E^{-1} =\ran E$ and $\ran E^{-1}=\dom E,$ then $E^{-1}$  may not even be a semi-projection.
\end{example}

\begin{corollary} \label{corid} \label{corid2}\label{coridsub}\label{coridsup}  Let $E \in \lr(\HH).$ The following are equivalent:
\begin{enumerate}
\item $E \in \Id(\HH)$;
\item $E=P_{\ker(I-E),\ker E} \ \hat{+} \ (\{0\} \times \mul E)=P_{\ran E, \ran(I-E)} \cap (\dom E \times \HH)$;
\item $E=P_{\ran E, \ran(I-E)}  \cap \ (\dom E \times \HH)$ and $\dom E = \ran E \cap \dom E +\ker E$;
\item $E=P_{\ker(I-E), \ker E} \  \hat{+} \ (\{0\} \times \mul E)$ and $\mul E \cap \dom E =\ran E \cap \ker E.$
\end{enumerate}
\end{corollary}

\begin{proof}
Apply Propositions \ref{sub}, \ref{supra} and  \ref{lemaIC}.
\end{proof}

\begin{corollary} \label{kerran} \label{kerran2} Let $E  \in \Id(\HH).$ Then
\begin{enumerate}
\item [1.] $\dom E=\ker(I-E)+\ker E.$
\item [2.] $\ran E=\ker(I-E)+\mul E.$
\item [3.] $\ker E=\ran(I-E) \cap \dom(E). $
\item [4.] $\mul E= \ran E \cap \ran(I-E).$
\end{enumerate}
\end{corollary}
\begin{proof} Use Propositions \ref{sub} and \ref{supra}.
\end{proof}

\begin{corollary} \label{IDset} The set of idempotent linear relations can be expressed as
$$
\{P_{\M \cap \St, \N \cap \St} \ \hat{+} \ (\{0\} \times (\M \cap \N))\}.$$ Alternatively,
$$\{P_{\M + \St, \N + \St} \cap  ((\M+\N) \times \HH)\}.$$
\end{corollary}
\begin{proof}  If $E \in \Id(\HH)$ then $E$ is sub-idempotent and, by Corollary \ref{subcor1}, $E=\PMN \cap(\St \times \HH)$ for some  subspaces $\M, \N, \St \subseteq \HH.$  Then, by Theorem \ref{corsubsup} and Lemma \ref{propmul}, 
$E=E^2=P_{\ker(I-E), \ker E} \ \hat{+} \ (\{0\} \times \mul E)=P_{\M \cap \St, \N \cap \St} \ \hat{+} \ (\{0\} \times (\M \cap \N)).$

Conversely, if $E=P_{\M \cap \St, \N \cap \St} \ \hat{+} \ (\{0\} \times (\M \cap \N))$  for some subspaces $\M, \N, \St \subseteq \HH$  then, by Corollary \ref{supcor1}, $E$ is super-idempotent. By Lemma \ref{propmul}, $\mul E \cap \dom E= \M \cap \N \cap \St =\ran E \cap \ker E.$ Therefore, by Proposition \ref{lemaIC}, $E \in \Id(\HH).$ The second equality follows in a similar way.
\end{proof}

\begin{proposition} \label{propidem} Let $E  \in \lr(\HH).$ Then
$E\in \Id(\HH)$ if and only if  $\dom E \subseteq \ran E+\ker E$ and $I_{\ran E \cap \dom E} \subseteq E.$ 
\end{proposition}

\begin{proof} If $E \in \Id(\HH)$ then, by Proposition \ref{supra}, $\dom E \subseteq \ran E+\ker E$ and, by Proposition \ref{sub}, $I_{\ran E \cap \dom E} \subseteq E.$ 
	
Conversely, since $I_{\ran E \cap \dom E} \subseteq E,$  $P_{\ran E \cap \dom E, \ker E} \subseteq E.$ By Proposition \ref{sub}, $E^2 \subseteq E.$ If $\dom E \subseteq \ran E+\ker E$ then $\dom E = \ran E \cap \dom E+\ker E.$ Therefore, by Proposition \ref{lemaIC}, $E^2=E.$
\end{proof}

\subsection{The idempotency condition}

This subsection is devoted to get a representation of idempotent relations similar to the representation of semi-projections (\ref{semiproy-rep}).

\begin{proposition} \label{min} There exists $E \in \Id(\HH)$ such that 
\begin{equation} \label{minc}
 \M\subseteq \ker(I-E), \  \N \subseteq \ker E  \mbox{ and } \St \subseteq \mul E.
\end{equation}  
Moreover, $E_0:=P_{\M+\St, \N+\St}  \cap ((\M+\N)\times \HH)$ is the smallest idempotent satisfying \eqref{minc}.
\end{proposition}

\begin{proof} 
By Lemma \ref{propmul}, $\M \subseteq (\M+\St) \cap (\M+\N)= \ker(I-E_0),$ $ \N \subseteq (\N+\St)\cap (\N+\M) = \ker E_0$ and $\St\subseteq (\St+\M) \cap (\St+\N) = \mul E_0.$ By  Corollary \ref{subcor1}, $E_0$ is sub-idempotent. 

It is easy to check that $E_0=P_{(\M+\N) \cap (\M+\St),\N} \ \hat{+} \ (\{0\}\times \St)$ so that, by Corollary \ref{supcor1}, $E_0$ is  super-idempotent. Then the idempotent $E_0$ satisfies \eqref{minc}.

Now, if $E \in \Id(\HH)$  satisfies (\ref{minc}) then 
$(\M+\St) \cap (\M+\N) \subseteq \ran E \cap \dom E$ so that, by Proposition \ref{propidem}, $I_{(\M+\N) \cap (\M+\St)}  \subseteq E.$ Since $\N \times \St \subseteq \ker E \times \mul E,$ we get that $E_0 \subseteq E.$
\end{proof}

\begin{proposition} \label{max} Let $\mc X, \mc Y, \mc Z$ be subspaces of $\HH$. Then there exists $F \in \Id(\HH)$ such that 
\begin{equation} \label{minc2}
\ran F \subseteq \mc X, \  \ran(I-F) \subseteq \mc Y \mbox{ and } \dom F\subseteq \mc Z.
\end{equation}  
Moreover, $F_0:=P_{\X\cap\Z, \Y\cap\Z}  \ \hat{+} \ ((\X\cap \Y)\times \HH)$ is the largest idempotent satisfying \eqref{minc2}.
\end{proposition}

\begin{proof}
 By Lemma \ref{propmul}, $\ran F_0= \X\cap\Z \cap (\Y\cap \Z+ \X\cap \Y) \subseteq \X,$ $\ran(I-F_0)=\Y\cap \Z \cap (\X \cap \Z+\X\cap \Y) \subseteq \Y$ and $\dom F_0 =(\X\cap\Z+\Y\cap \Z)\cap \X\cap\Y\subseteq \Z.$ By Corollary \ref{supcor1}, $F_0$ is  super-idempotent.
 It is easy to check that $F_0=P_{\X,\Y} \cap  ((\Z\cap\X+\Z\cap\Y) \times \HH)$ so that,  by  Corollary \ref{subcor1}, $F_0$ is sub-idempotent. Then the idempotent $F_0$ satisfies \eqref{minc2}. 

Now, if $F \in \Id(\HH)$  satisfies (\ref{minc2}) then, by Corollary \ref{kerran}, $\ran F \cap \dom F \subseteq \mc X \cap \mc Z,$ $\ker F =\ran(I-F)\cap \dom F \subseteq \mc Y \cap \mc Z$ and $\mul F =  \ran F \cap \ran (I-F) \subseteq \mc X \cap \mc Y.$ Then,  by Corollary \ref{corid}, $F=P_{\ran F \cap \dom F, \ker F}  \ \hat{+} \ \ (\{0\}  \times \mul F) \subseteq F_0.$ 
\end{proof}

In what follows we characterize the triplets  for which there is equality in \eqref{minc} or in \eqref{minc2}.

\begin{proposition} \label{theic} \label{subidem}\label{unicidadE}
The following are equivalent:
\begin{enumerate}
\item There exists $E\in \Id(\HH)$ such that $\M=\ker(I-E),$ $\N=\ker E$ and $\St=\mul E;$
\item  $(\M+\N)\cap \St=\M \cap \N$;
\item	$(\M+\N)\cap (\M+\St)=\M$ and $\M\cap \N=\M\cap \St.$ 
\end{enumerate}
In this case, $E=\PMN \ \hat{+} \ (\{0\} \times \St)$ is the unique idempotent satisfying item $i).$ %$\ker(I-E)=\M,$ $\ker E=\N$ and $\mul E=\St.$ 
\end{proposition}
\begin{proof} 
$i) \Rightarrow ii)$: Applying Corollary \ref{kerran} to $E$ and $I-E,$ it follows that $(\M+\N)\cap \St=(\ker(I-E)+\ker E)\cap \mul E=\dom E \cap \mul E=\dom E \cap \ran(I-E)\cap \ran E=\ker(I-E) \cap \ker E=\M \cap \N.$

$ii) \Rightarrow iii)$:  Suppose that $(\M+\N)\cap \St=\M \cap \N.$ Then $\M \cap \N \subseteq \St.$ So that $\M \cap \N=\St \cap \M \cap \N:= \mathcal{W}.$ From $\M \cap \St+\N \cap \St \subseteq (\M+\N) \cap \St=\M \cap \N,$ we get $\M \cap \St \subseteq\N$ and $\N \cap \St \subseteq \M.$ Therefore, $\M \cap \St= \N \cap \St=\mc{W}.$ Let $x \in (\M+\N)\cap (\M+\St),$ so that  $x=m+n=m'+s$ with $m,m' \in \M,$ $n\in \N$ and $s\in\St.$ Then
$m+n-m'=s \in (\M +\N) \cap \St=\M \cap \N$ and $x=m'+s \in \M +\M\cap \N=\M.$ The other inclusion always holds. 

$iii) \Rightarrow i)$:  
Define $E:=\PMN \ \hat{+} \ (\{0\} \times \St).$ Then $\dom E=\M+\N$ and $\ran E=\M+\St.$ So that $\ran E\cap\dom E=(\M+\St)\cap(\M+\N)=\M.$ By Lemma \ref{propmul}, $\ker E=\M \cap \St + \N=\M\cap \N+\N=\N$ and $\mul E=\M \cap \N+ \St=\M \cap \St+ \St=\St.$ Then, by Corollary \ref{coridsup}, $E \in \Id(\HH).$
Finally, if $E_1$ satisfies $(i)$ then, by Proposition \ref{min}, $E\subseteq E_1.$ Since $\dom E=\M+\N=\dom E_1$ and $\mul E=\St=\mul E_1,$ by Lemma \ref{lemalr}, $E=E_1$.
\end{proof}

If  $(\M+\N)\cap \St=\M \cap \N$, it is easy to check that any triplet having $\M,\N$ and $\St$ as components satisfies the corresponding equality (see Corollary \ref{coric}). 

\begin{proposition} \label{theic2} \label{subidem2}\label{unicidadE2}
Let $\X, \Y,\Z$ be subspaces of $\HH.$ The following are equivalent:
\begin{enumerate}
\item There exists $F\in \Id(\HH)$ such that $\X=\ran F,$ $\Y=\ran (I-F)$ and $\Z=\dom F;$
\item $\X\cap \Y+\Z=\X +\Y;$
\item  $\X=\X\cap \Y +\X \cap \Z$ and $\X+\Y=\X+\Z.$ 
\end{enumerate}
In this case, $F=\PXY \cap (\Z \times \HH)$ is the unique idempotent satisfying item $i).$ 
\end{proposition}
\begin{proof} $i) \Rightarrow ii)$: By Corollary \ref{kerran}, $\X \cap \Y+\Z=\ran F \cap \ran(I-F)+\dom F=\mul F+ \dom F=\mul F+ \ran F \cap \dom F+ \ker F=\ran F +\ker F=\X+\Y.$
	
	$ii) \Rightarrow iii)$: $\X+\Y=\X\cap \Y+\Z \subseteq \X + \Z \subseteq \X+\Y$ because $\Z \subseteq \X+\Y.$ Then $\X+\Y=\X+\Z.$
	Since $\X \subseteq \X\cap \Y+\Z$ then $\X \subseteq (\X\cap \Y+\Z)\cap \X=\X\cap \Y+\X\cap \Z.$ The other inclusion always holds.
	
	$iii) \Rightarrow i)$:  Define $F:=\PXY \cap (\Z \times \HH).$ Since $\Z \subseteq \X+\Y,$ $\Y \subseteq \X+\Z,$ and $\X \subseteq \Y+\Z,$ it follows that $\dom F=(\X+\Y)\cap \Z=\Z,$ $\ran F=\X\cap (\Y+\Z)=\X$ and $\ran(I-F)=\Y \cap (\X+\Z)=\Y.$ Also, $\Z \subseteq \X+\Y = \X \cap \Y +\X \cap \Z+\Y=\X\cap \Z+\Y.$ Then $\Z=\X\cap \Z+\Y\cap \Z,$ so that $\dom F= \ran F \cap \dom F+\ker F.$ 
	Therefore, by Corollary \ref{coridsub}, $F \in \Id(\HH).$
	
	Finally, if $F_1$ satisfies $(i)$ then, by Proposition \ref{min}, $F_1 \subseteq F.$ Since $\dom F=\Z=\dom F_1$ and, by Corollary \ref{kerran}, $\mul F=\X \cap \Y=\mul F_1,$ then $F=F_1$.
\end{proof}

It is easy to see that item $iii)$ of Proposition \ref{theic} is equivalent to
\begin{align*}
\M&=(\M+\N)\cap (\M+\St),\\
\N&=(\N+\M)\cap (\N+\St),\\
\St&=(\St+\M)\cap (\St+\N).
\end{align*}
In a symmetric fashion item $iii)$ of Proposition \ref{theic2} is equivalent to
\begin{align*}
\X&=\X\cap \Y +\X \cap \Z, \\ 
\Y&=\Y\cap \X +\Y \cap \Z\\
\Z&=\Z\cap \X +\Z \cap \Y.
\end{align*}

In view of Propositions \ref{subidem} and \ref{theic2}, the set of idempotent linear relations can  be given in terms of subspaces.

\begin{corollary} \label{IP1} The set of idempotent linear relations can be expressed as
$$\{\PMN  \hat{+} (\{0\} \times \St) :  (\M+\N)\cap \St=\M\cap \N\}.$$ Alternatively;
$$\{\PXY \cap (\Z \times \HH) : \X,\Y, \Z \subseteq \HH \text{ subspaces, }  \X \cap \Y+\Z=\X+\Y\}.$$
\end{corollary}

\begin{proposition} \label{the IC equiv}  If
\begin{equation} \label{ic1}
(\M+\N)\cap \St=\M \cap \N
\end{equation}
then $$\X:=\M+\St, \ \Y:=\N+\St \mbox{ and } \Z:=\M+\N$$ satisfy 
\begin{equation} \label{ic2}
\X\cap \Y+\Z=\X +\Y.
\end{equation}
Conversely, if the subspaces  $\X, \Y,\Z$ satisfy \eqref{ic2} then $$\M:=\X \cap \Z, \ \N:=\Y \cap \Z \mbox{ and } \St:=\X\cap \Y$$ satisfy \eqref{ic1}.
\end{proposition}

\begin{proof} By Proposition \ref{theic}, $\X \cap \Y=(\M+\St)\cap (\N+\St)=\St.$ Therefore $\X,\Y, \Z$ satisfiy \eqref{ic2}.
Conversely, by Proposition \ref{theic2}, $\M+\N=\X \cap \Z+\Y \cap \Z=\Z.$ Therefore $\M,\N, \St$ satisfy \eqref{ic1}.
\end{proof}

Summarizing, by Propositions \ref{theic} and \ref{theic2}, we get that $E\in \Id(\HH)$ is characterized by any of the following triplets:  $$\ker(I-E), \ker E  \mbox{ and } \mul E$$ or  $$\ran E, \ran(I-E) \mbox{ and } \dom E,$$
and Proposition \ref{the IC equiv} shows how to get one triplet from the other. Any of these triplets provides a unique representation of an $E\in \Id(\HH);$ namely, $$E=P_{\ker(I-E),\ker E} \ \hat{+} \ (\{0\} \times \mul E)$$ or $$E=P_{\ran E, \ran(I-E)} \cap (\dom E \times \HH).$$  The first representation of $E\in\Id(\HH)$ resembles the representation (\ref{semiproy-rep}) of semi-projections. So, from now on, we use this representation rather than the second one. 
 
\begin{definition} The subspaces $\M, \N,\St$ satisfy the \emph{idempotency condition (IC)} if $$(\M+\N)\cap \St=\M \cap \N.$$
\end{definition}

\begin{example} \label{exampleIC}
	\begin{enumerate}
\item [1.] $\M,\N,\M\cap\N$ satisfy the IC.
\item [2.] $(\M+\N)\cap (\M+\St),$ $(\N+\M)\cap(\N+\St)$ and $(\St + \M)\cap(\St+\N)$ are the ``minimal'' subspaces of those with the IC containing $\M,\N$ and $\St,$ respectively (see  Proposition  \ref{min}).
\item [3.] $\N \cap\St,\M\cap \St$ and $\M\cap\N$ are the ``maximal'' subspaces of those with the IC contained in $\M,\N$ and $\St,$ respectively (see Proposition  \ref{max}).
\item [4.] If $(\M\dot{+}\N)\cap \St=\{0\}$ then $\M,\N$ and $\St$ satisfy the IC.
\item [5.]  If $\ol{\M+\N}\cap \ol{\St}=\ol{\M}\cap \ol{\N}$ then $\ol{\M},\ol{\N}$ and $\ol{\St}$ satisfy the IC. 
\item[6.] If $T \in \lr(\HH)$ then $\ran T \cap \dom T,$ $\ker T$ and $\mul T$ satisfy the IC if and only if $\mul T \cap \dom T=  \ran T \cap \ker T.$ 
	\end{enumerate}
\end{example}

\bigskip
In the sequel, given $\M, \N,$ $\St$ subspaces of $\HH$  satisfying the IC, we write
$$P_{\M,\N,\St}:=\PMN  \ \hat{+} \ (\{0\} \times \St).$$
 In other words, $P_{\M,\N,\St}$ is the unique idempotent with $\ker(I-P_{\M,\N,\St})=\M$, $\ker P_{\M,\N,\St}=\N$ and $\mul P_{\M,\N,\St}=\St.$ We
emphasize here that throughout this paper, the notation $P_{\M,\N,\St}$ is used for $\PMN  \ \hat{+} \ (\{0\} \times \St)$  only when $\M, \N,$ $\St$ satisfy the IC. 
By Proposition \ref{the IC equiv}, 
$$P_{\M,\N,\St}=P_{\M+\St, \N+\St}  \cap ((\M+\N) \times \HH).$$

In particular, $P_{\M,\N}=P_{\M,\N,\M\cap \N}.$ 

\begin{corollary} \label{coric} Let $\M, \N,\St$ be subspaces of $\HH$ satisfying the IC. Then 
$P_{\M,\N,\St}^{-1}=P_{\M,\St,\N}$ and $I-P_{\M,\N,\St}=P_{\N,\M,\St}.$ 
\end{corollary}

\section{The closure and adjoint}

This section is devoted to study the closure and the adjoint of sub-, super- and idempotent relations. The adjoint and the closure of semi-projections are again semi-projections. The following formulae for $E^*$ and $\ol{E}$ where proved in \cite{Cross} and \cite{Labrousse}.

\begin{proposition} \label{L2} If $E=P_{\M, \N}$ then  $$E^*=P_{\N^{\perp},\M^{\perp}} \mbox{ and } \ol{E}=P_{\ol{\M}, \ol{\N}}.$$ Moreover, $E$ is closed if and only if $\M$ and $\N$ are closed.
\end{proposition}

\begin{example} Let $\M$ and $\N$ be subspaces of $\HH$ such that $\M \cap \N=\{0\}$ but $\ol{\M}\cap \ol{\N}\not = \{0\}$ and $\M^{\perp} \cap \N^{\perp}\not =\{0\}.$ Let $E$ be the projection onto $\M$ with kernel $\N.$ Then, by Proposition \ref{L2}, $E^*$ and $\ol{E}$ are both semi-projections but they are not projections.
\end{example}

We begin by studying the closure and the adjoint of sub- and super-idempotents.

\begin{lemma} \label{Tclosed} Let $\M, \N, \St$ be closed subspaces of $\HH.$  Then
	\begin{enumerate}
		\item[1.] $\PMN \cap  (\St \times \HH)$ is closed.
		\item[2.] $P_{\M,\N} \ \hat{+} \ (\{0\} \times \St)$ is closed  if and only if $\M^{\perp}+\N^{\perp}+\St^{\perp}$ is closed.
	\end{enumerate} 
\end{lemma}

\begin{proof}  By Proposition \ref{L2}, $P_{\M,\N}$ is closed. Since $\St \times \HH$ is closed, item $1$ follows. 

By Lemma \ref{sumclosed}, $P_{\M,\N} \ \hat{+} \ (\{0\} \times \St)$ is closed if and only if $$P_{\M,\N}^* \ \hat{+} \ (\{0\} \times \St)^*=P_{\N^{\perp}, \M^{\perp}} \ \hat{+} \ (\St^{\perp} \times \HH)$$ is closed. But it is easy to check that
$$P_{\N^{\perp}, \M^{\perp}} \ \hat{+} \ (\St^{\perp} \times \HH)=(\M^{\perp}+\N^{\perp}+\St^{\perp}) \times \HH.$$
The latter is closed if and only if $\M^{\perp}+\N^{\perp}+\St^{\perp}$ is closed.
\end{proof}

\begin{corollary} \label{Tclosed2} $$\ol{P_{\M,\N} \ \hat{+} \ (\{0\} \times \St)}=P_{\ol{\M},\ol{\N}} \ \hat{+} \ (\{0\} \times \ol{\St})$$ if and only if $\M^{\perp}+\N^{\perp}+\St^{\perp}$ is closed.
\end{corollary}

\begin{lemma} \label{LemaOR} Let $\M, \N, \St$ be operator ranges in $\HH$ such that $\M+\N+\St$ is closed. Then
$$\ol{\PMN \cap  (\St \times \HH)}=P_{\ol{\M},\ol{\N}} \cap (\ol{\St}\times \HH).$$
\end{lemma}
\begin{proof} Consider $T_1=\PMN $ and $T_2=\St \times \HH.$ Then $T_1$ and $T_2$ are operator ranges (see \cite{Filmore}) and $T_1 \hat{+}T_2=\PMN \ \hat{+} \  (\St \times \HH)=(\M+\N+\St) \times \HH$ is closed because $\M+\N+\St$ is closed. Then, applying Propositions \ref{rangeop} and  \ref{L2}, $\ol{\PMN \cap  (\St \times \HH)}=\ol{\PMN} \cap \ol{\St\times \HH}=P_{\ol{\M},\ol{\N}} \cap (\ol{\St}\times \HH).$
\end{proof}

 In the next result we characterize the super-idempotents that are closed and, in particular, the closed idempotent relations.

\begin{proposition}\label{Ecerrado}  Let $E$ be super-idempotent and set $\M:=\ker(I-E), \N=\ker E$ and $\St=\mul E.$  Then the following are equivalent:
\begin{enumerate}
\item $E$ is closed;
\item $\M, \N, \St$ and $\M^{\perp}+\N^{\perp}+\St^{\perp}$ are closed.
\end{enumerate}
\end{proposition}

\begin{proof} By Proposition \ref{supra}, $E=P_{\M,\N} \ \hat{+} \ (\{0\} \times \St).$
If $E$ is closed,  $\N=\ker E$ and $\St=\mul E$ are closed. Also, since $I$ is bounded, $\ol{I-E}=(I-E)^{**}=I-E^{**}=I-\ol{E}=I-E.$ So that $I-E$ is closed and $\M=\ker(I-E)$ is  closed. Then, by Lemma \ref{Tclosed},  $\M^{\perp}+\N^{\perp}+\St^{\perp}$ is closed.
The converse follows applying Lemma \ref{Tclosed} once again.
\end{proof}

As a corollary we get a characterization of the closed idempotents. 
\begin{theorem} Let $E=P_{\M,\N,\St}.$ Then $E$ is closed if and only if $\M, \N, \St$ and $\M^{\perp}+\N^{\perp}+\St^{\perp}$ are closed.
\end{theorem}

Our next goal is to study the adjoint and closure of idempotent relations. In general, these operations are not closed on $\Id(\HH)$ (see Examples \ref{starnoid} and \ref{exclausura}).

\begin{proposition} \label{propestrella} 
$$(P_{\M,\N} \ \hat{+} \ (\{0\} \times \St))^*=P_{\N^{\perp},\M^{\perp}} \cap (\St^{\perp} \times \HH).$$
\end{proposition}

\begin{proof} Apply \eqref{sumadj} and Proposition \ref{L2}.
\end{proof}

By the above proposition, we get that the adjoint of a super-idempotent is always sub-idempotent. However, a similar statement is no longer valid if we interchange the sub- and super-idempotent condition. 

\begin{proposition} \label{propestar1}
	$$(\PMN \cap (\St \times \HH))^*=P_{\N^{\perp},\M^{\perp}} \ \hat{+} \ (\{0\} \times\St^{\perp})$$ if and only if
	$$\ol{\M}+\ol{\N}+\ol{\St} \mbox{ is closed and } \ol{\PMN \cap (\St \times \HH)}=P_{\ol{\M},\ol{\N}} \cap (\ol{\St}\times \HH).$$
\end{proposition}

\begin{proof} Set $R:=\PMN \cap (\St \times \HH).$ If $R^*=P_{\N^{\perp},\M^{\perp}} \ \hat{+} \ (\{0\} \times\St^{\perp})$ then, by Corollary \ref{Tclosed2}, $\ol{\M}+\ol{\N}+\ol{\St}$ is closed. 
	Also, $\ol{R}=R^{**}=P_{\ol{\M},\ol{\N}} \cap (\ol{\St}\times \HH).$
	
	Conversely, if $\ol{R}=P_{\ol{\M},\ol{\N}} \cap (\ol{\St}\times \HH)$ then, by Proposition \ref{L2}, $$R^*=\ol{P_{\N^{\perp},\M^{\perp}} \ \hat{+} \ (\{0\} \times\St^{\perp})}=P_{\N^{\perp},\M^{\perp}} \ \hat{+} \ (\{0\} \times\St^{\perp}),$$ because $\ol{\M}+\ol{\N}+\ol{\St}$ is closed. 
\end{proof}

\begin{corollary} Let $\M,$ $\N$ and $\St$ be operator ranges of $\HH$ such that $\M+\N+\St$ is closed. Then 
$$(P_{\M,\N} \cap (\St\times \HH))^*=P_{\N^{\perp},\M^{\perp}} \ \hat{+} \ (\{0\} \times\St^{\perp}).$$
\end{corollary}
\begin{proof} Apply Lemma \ref{LemaOR} and Proposition \ref{propestar1}.
\end{proof}

\begin{corollary} Let $E \in \lr(\HH)$ be sub-idempotent. Then $E^*$ is super-idempotent if and only if $\clran E+\clran(I-E)+\cldom E$ is closed and $\ol{E}=P_{\clran E, \clran(I-E)} \cap (\cldom E \times \HH).$ 
\end{corollary}

\begin{proof} If $E^*$ is super-idempotent then, by Proposition \ref{supra}, $$E^*=P_{\ker(I-E^*), \ker E^*} \ \hat{+} \ (\{0\}\times \mul E^*)=P_{\ran(I-E)^{\perp}, \ran E^{\perp}} \ \hat{+} \ (\{0\}\times \dom E^{\perp}).$$ Then, by Lemma \ref{Tclosed}, $\clran E+\clran(I-E)+\cldom E$ is closed and, by \eqref{sumadj} and Proposition \ref{L2}, $\ol{E}=P_{\clran E, \clran(I-E)} \cap (\cldom E \times \HH).$ 
Conversely, since $E$ is sub-idempotent, by Proposition \ref{sub}, $E=P_{\ran E, \ran(I-E)} \cap (\dom E \times \HH).$ Then, by Proposition \ref{propestar1}, $E^*=P_{\ran(I-E)^{\perp}, \ran E^{\perp}} \ \hat{+} \ (\{0\}\times \dom E^{\perp}).$ Then, by Proposition \ref{supra}, $E^*$ is super-idempotent.
\end{proof}

As a corollary of Proposition \ref{propestar1} we get the following characterization of those idempotents admitting an idempotent adjoint.

\begin{theorem} \label{PropStarcl} Let $E =P_{\M,\N,\St}.$ Then $E^* \in \Id(\HH)$ if and only if $$\ol{E}=P_{\ol{\M+\St}, \ol{\N+\St}} \cap ((\ol{\M+\N})\times \HH)$$ and $\ol{\M+\St}+\ol{\N+\St}+\ol{\M+\N}$ is closed.
\end{theorem}
\begin{proof} 
	The result follows applying Proposition \ref{propestar1} to $E=P_{\M+\St, \N+\St} \cap ((\M+\N)\times \HH).$
\end{proof}

\begin{proposition} \label{estrella}  Let $E =P_{\M,\N,\St}.$
Then the following are equivalent:
\begin{enumerate}
	\item $E^* \in \Id(\HH);$
	\item $E^*=P_{\N^{\perp}\cap \St^{\perp},\M^{\perp}\cap \St^{\perp},\M^{\perp}\cap \N^{\perp}};$
	\item $(\N^{\perp}+\M^{\perp})\cap \St^{\perp} =\N^{\perp}  \cap \St^{\perp}+ \M^{\perp}  \cap \St^{\perp}.$
\end{enumerate}
\end{proposition}
\begin{proof} Since $P_{\ker(I-E^*), \ker E^*} \ \hat{+} \ (\{0\} \times \mul E^*)=P_{\N^{\perp}\cap \St^{\perp},\M^{\perp}\cap \St^{\perp},\M^{\perp}\cap \N^{\perp}}$
and $E^*$ is sub-idempotent, the result follows from Proposition \ref{lemaIC} and Corollary \ref{corid}.
\end{proof}

\begin{corollary} \label{corclausura} Let $E =P_{\M,\N,\St}.$ Then the following are equivalent:
	\begin{enumerate}
		\item $E^*$ and  $\ol{E} \in \Id(\HH);$
		\item $(\N^{\perp}+\M^{\perp})\cap \St^{\perp} =\N^{\perp}  \cap \St^{\perp}+ \M^{\perp}  \cap \St^{\perp}$ and \\
		$(\ol{\M+\St}+\ol{\N+\St}) \cap \ol{\M+\N}=\ol{\M+\St}\cap \ol{\M+\N}+\ol{\N+\St}\cap \ol{\M+\N};$
		\item $E^*=P_{\N^{\perp}\cap \St^{\perp},\M^{\perp}\cap \St^{\perp},\M^{\perp}\cap \N^{\perp}}$ and \\ $\ol{E}=P_{\ol{\M+\St}\cap \ol{\M+\N}, \ol{\N+\St} \cap \ol{\M+\N}, \ol{\M+\St}\cap \ol{\N+\St}}.$
	\end{enumerate} 
\end{corollary}

Using results about the adjoint of linear relations \cite{Sandovici} and operator ranges \cite{Labrousse1980}, we give examples of closed idempotent linear relations with idempotent adjoint and idempotent linear relations admitting idempotent adjoint and idempotent closure.

\begin{proposition} \label{corexSando} Let $E =P_{\M,\N,\St}$ such that $\M+\St$ and $\N$ are closed. 
If $$\N^{\perp} \cap \St^{\perp}+\N^{\perp}\cap \M^{\perp}=\N^{\perp}$$ then $E$ is closed and $E^*\in \Id(\HH).$
\end{proposition}
\begin{proof} It follows by Theorem \ref{sando1} applied to $A=P_{\N^{\perp}\cap \St^{\perp},\M^{\perp}\cap \St^{\perp},\M^{\perp}\cap \N^{\perp}}$ and $B=P_{\M,\N,\St}.$
\end{proof}

%\begin{corollary} \label{corexSando2} Let $E =P_{\M,\N,\St}$ such that $\M+\N$ and $\St$ are closed. 
%	If $$\St^{\perp} \cap \N^{\perp}+\St^{\perp}\cap \M^{\perp}=\St^{\perp}$$ then $E$ is closed and $E^*\in \Id(\HH).$
%\end{corollary} 

\begin{proposition} \label{corex} Let $E =P_{\M,\N,\St}$ where $\M,$ $\N$ and $\St$ are operator ranges.  If $\M+\N+\St$ is closed then $E^* \in \Id(\HH)$ and $\ol{E} \in \Id(\HH).$
\end{proposition}

\begin{proof} If $\M+\N+\St$ is closed, since $\St=(\M+\St)\cap (\N+\St),$ applying Proposition \ref{rangeop} to $\M+\St$ and $\N+\St,$ we get that 
\begin{equation} \label{eqOPR1}
\St^{\perp}=\M^{\perp}\cap\St^{\perp}+\N^{\perp}\cap \St^{\perp}.
\end{equation}
Then $(\M^{\perp}+\N^{\perp})\cap \St^{\perp} \subseteq \St^{\perp}=\M^{\perp}\cap\St^{\perp}+\N^{\perp}\cap \St^{\perp}.$ Therefore, $(\M^{\perp}+\N^{\perp})\cap \St^{\perp} =\M^{\perp}  \cap \St^{\perp}+ \N^{\perp}  \cap \St^{\perp}$ and, by Proposition \ref{estrella}, $E^* \in \Id(\HH).$ 

Applying again Proposition \ref{rangeop} to $\M+\N$ and $\St,$ it follows that $$((\M+\N)\cap \St)^{\perp}=\M^{\perp}\cap \N^{\perp}+\St^{\perp}.$$ So that $\M^{\perp}\cap \N^{\perp}+\St^{\perp}$ is closed. Therefore,  by \eqref{eqOPR1}, $\M^{\perp}\cap\N^{\perp}+\M^{\perp}\cap\St^{\perp}+\N^{\perp}\cap \St^{\perp}$ is closed. Then, from the same arguments in the first part of the proof applied to $E^*,$ it follows that $\ol{E} \in \Id(\HH).$
\end{proof}

\begin{remark} By Proposition \ref{corex}, if $\HH$ is finite-dimensional then the adjoint of every idempotent linear relation is idempotent.
\end{remark}

Now, we are in a position to give an example of an idempotent $E$ such that $E^*\notin \Id(\HH)$.

\begin{example}  \label{starnoid} Let $\mc X$ be an infinite dimensional Hilbert space and set $\HH:=\mc X \times \mc X \times \mc X.$ Take $\M:=\mc X \times \{0\} \times \{0\}$ and $\N:= \{0\} \times \mc X \times \{0\}.$ Let $\mc Z:=\{0\} \times \{0\} \times \mc X$ and $\mc W:=\{(x,x,0) : x \in \mc X\} \subseteq \M \dotplus \N.$ Following similar arguments as those found in \cite[page 28]{Halmos}, we can construct a closed subspace $\St$ such that $\St \cap \mc W=\{0\}$ and $c_0(\St, \mc W)=1,$ so that $\St \dotplus \mc W$ is not closed and $\St \subseteq   \mc W +  \mc Z:=\Pi.$ 
	
The subspace $\M \dotplus \N=\mc X \times \mc X \times \{0\}$ is closed and $\M \dotplus \Pi=\HH.$ In fact, if $x \in \M \cap \Pi$ then there exists $\alpha \in \mc X$ such that $x=(\alpha, 0,0) \in \Pi.$ But also $x=(\beta, \beta, \gamma)$ with $\beta, \gamma \in \mc X.$ So that $\alpha=\beta=0=\gamma$ and $x=0.$ Also, given $x=(\alpha, \beta,\gamma) \in \mc \HH$ then $x=(\alpha-\beta,0,0) + (\beta,\beta,\gamma)\in \M \dotplus \Pi.$ 
	
Now, let us see that $\M \dotplus \St$ is closed. In fact, $\M \cap \St \subseteq \M \cap \Pi=\{0\}.$ Then $\M \dotplus \St \subseteq \M \dotplus \Pi$ so that, 
$$c_0(\M,\St) \leq c_0(\M,\Pi)<1.$$ Hence  $\M \dotplus \St$ is closed. 
In a similar way, $\N \dotplus \St$ is closed. 
	
Also $\M, \N$ and $\St$ satisfy the IC: in fact $(\M \dotplus \N)\cap \St=(\M \dotplus \N)\cap \St \cap \Pi=\mc W \cap \St=\{0\}=\M \cap \N,$ where we used that $(\M \dotplus \N) \cap \Pi=\mc W.$ Set $E:=P_{\M,\N,\St}.$ Since $\M, \N, \St$ are closed and $\M^{\perp}+\N^{\perp}+\St^{\perp}=\HH,$  $E$ is closed, by Proposition \ref{Ecerrado}.
Now, since $\St \dotplus \mc W \subseteq \St \dotplus (\M \dotplus \N),$ it follows that $1=c_0(\St, \mc W) \leq c_0(\St, \M\dotplus \N).$ Then $c_0(\St, \M\dotplus \N)=1$ and $\M+\N+\St$ is not closed. But, since $\M\dotplus \N,$ $\M\dotplus \St$ and $\N \dotplus \St$ are closed, if $E^* \in \Id(\HH)$ then, by Proposition \ref{estrella}, $E^*=P_{\N^{\perp} \cap \St^{\perp},\M^{\perp} \cap \St^{\perp},\M^{\perp} \cap \N^{\perp}}$ and, by Proposition \ref{Ecerrado}, $\ol{\N+\St}+\ol{\M+\St}+\ol{\M+\N}=\M+\N+\St$ is closed, which is absurd. Hence $E^* \not \in \Id(\HH).$
\end{example}

The same example provides an idempotent $F$ such that $\ol{F}\notin \Id(\HH)$ and also a linear relation $T \not \in \Id(\HH)$ such that $T^* \in \Id(\HH).$

\begin{example} \label{exclausura} Let $\HH$ and $\M, \N, \St$ be as in Example \ref{starnoid}. Set $E:=P_{\M,\N,\St}$ and $F:=P_{\N^{\perp} \cap \St^{\perp},\M^{\perp} \cap \St^{\perp},\M^{\perp} \cap \N^{\perp}}.$ Then $E, F \in \Id(\HH)$ and since $\M\dotplus \N,$ $\M\dotplus \St$ and $\N \dotplus \St$ are closed,
$$E=P_{\M+\St, \N+\St} \cap ((\M+\N) \times \HH)=P_{\ol{\M+\St}, \ol{\N+\St}} \cap (\ol{\M+\N} \times \HH)=F^*.$$
Hence $E^*=\ol{F}.$ By Example \ref{starnoid}, $\ol{F}  \not \in \Id(\HH).$

On the other hand, set $T:=\ol{F}$ then $T \not \in \Id(\HH)$ and $T^*=F^*=E \in \Id(\HH).$
\end{example}

\bigskip
If $E =P_{\M,\N,\St}$ then the inclusions 
\begin{equation} \label{inclusion}
P_{\ol{\M}, \ol{\N}} \ \hat{+} \ (\{0\} \times \ol{\St}) \subseteq \ol{E} \subseteq P_{\ol{\M+\St}, \ol{\N+\St}} \cap (\ol{\M+\N} \times \HH)
\end{equation}
always hold with equalities when $E$ is a semi-projection. In the next proposition we provide conditions to get equalities in (\ref{inclusion}).

\begin{proposition} \label{clausurares}  Let $E =P_{\M,\N,\St}.$ Then:
	\begin{enumerate}
		\item [1.] If  $E^* \in \Id(\HH)$ then $\ol{E} = P_{\ol{\M+\St}, \ol{\N+\St}} \cap (\ol{\M+\N} \times \HH).$
		\item [2.] $\ol{E}=P_{\ol{\M}, \ol{\N}} \ \hat{+} \ (\{0\} \times \ol{\St})$ if and only if $\M^{\perp}+\N^{\perp}+\St^{\perp} $ is closed.
		\item [3.] If $E^* \in \Id(\HH)$ and $\M^{\perp}+\N^{\perp}+\St^{\perp} $ is closed then there is equality in \eqref{inclusion} and $$\ol{E}=P_{\ol{\M} + \ol{\N}\cap \ol{\St}, \ol{\N}+\ol{\M}\cap \ol{\St}, \ol{\St}+\ol{\M}\cap\ol{\N}}.$$ Moreover,   $$\ol{E}=P_{\ol{\M}+\ol{\St}, \ol{\N}+\ol{\St}} \cap ((\ol{\M}+\ol{\N}) \times \HH).$$
	\end{enumerate} 
\end{proposition}

\begin{proof} $1$: Use Theorem \ref{PropStarcl}.

$2$: Use Corollary \ref{Tclosed2}.

$3$: From $1$ and $2,$ $\ol{E}$ is sub- and super-idempotent. Then $\ol{E} \in \Id(\HH)$ and Corollary \ref{corclausura} gives the formula for $\ol{E}.$ Finally, using Lemma \ref{propmul} we get that
$\ker \ol{E}=\ol{\N} + \ol{\M}\cap \ol{\St},$ $\ker(I-\ol{E})=\ol{\M}+\ol{\N}\cap \ol{\St}$ and $\mul \ol{E}=\ol{\St}+\ol{\M}\cap\ol{\N},$ $\ran \ol{E}=\ol{\M}+\ol{\St},$ $\ran(I-\ol{E})=\ol{\N}+\ol{\St}$ and $\dom \ol{E}=\ol{\M}+\ol{\N}.$
\end{proof}

If $E$ is a closed semi-projection then $E^*$ is a semi-projection, and $\ran E$ and $\ran(I-E)=\ker E$ are both closed. In general, this is no longer true for closed idempotents. In what follows, we characterize those closed idempotents $E$ with $\ran E$ and $\ran(I-E)$ closed such that $E^*$ is idempotent. In this case, by Theorem \ref{closedrange}, $\ran E^*$ and $\ran (I-E^*)$ are closed.

\begin{proposition} \label{closedstar2} Consider $E=P_{\M,\N,\St}$ such that E is closed. Then the following are equivalent:
	\begin{enumerate}
		\item $E^* \in \Id(\HH),$ $\ran E$ and $\ran(I-E)$ are closed;
		\item $\ol{\M+\N} +\St$ is closed and $\ol{\M+\N}\cap \St= \M\cap\N.$ 
	\end{enumerate} 
\end{proposition}

\begin{proof} 
If $i)$ holds, by Corollaries \ref{propidemL} and \ref{kerran2}, $\M+\St$ and $\N+\St$ are closed and, applying Theorem \ref{PropStarcl},  $\ol{\M+\N} +\St=\ol{\M+\N}+\ol{\N+\St}+\ol{\M+\St}$ is closed. On the other hand,  by Theorem \ref{PropStarcl} again, $E=\ol{E}=P_{\M+\St, \N+\St} \cap (\ol{\M+\N} \times \HH)$ then $\M+\N=\dom E=(\M+\N+\St) \cap \ol{\M+\N}=\M+\N+\St \cap \ol{\M+\N},$  so that $\St \cap \ol{\M+\N} \subseteq \M+\N.$ Then $\St \cap \ol{\M+\N}=\St \cap (\M+\N)=\M \cap \N,$ where we used that $\M, \N$ and $\St$ satisfy the IC.
	
	Conversely, if $ii)$ holds then $\ol{\M+\N} \cap \St = \M\cap\N=\M \cap \St=\N \cap \St,$ because $\M, \N$ and $\St$ satisfy the IC. Then, by Lemma \ref{lemaangulo} and Theorem \ref{prop angulos}, $$c(\M,\St)\leq c(\ol{\M+\N}, \St)<1.$$ So that, by Theorem \ref{prop angulos}, $\ran E=\M+\St$ is closed. Likewise, $\ran(I-E)=\N+\St$ is closed. 
	Therefore, $\ol{\M+\St}+\ol{\N+\St}+\ol{\M+\N}=\ol{\M+\N}+\St$ is closed. 	Also, since $E \subseteq P_{\M+\St, \N+\St} \cap (\ol{\M+\N} \times \HH),$ $\dom(P_{\M+\St, \N+\St} \cap (\ol{\M+\N} \times \HH))=\M+\N+\St \cap \ol{\M+\N}=\M+\N=\dom E$ and $\mul(P_{\M+\St, \N+\St} \cap (\ol{\M+\N} \times \HH))=\M\cap\St+\N \cap \St=\St=\mul E,$ using Theorem \ref{PropStarcl}, it follows that  $E^*\in \Id(\HH).$
\end{proof}

\begin{theorem} \label{thmclausura} Let $E =P_{\M,\N,\St}.$ Then the following are equivalent:
	\begin{enumerate}
		\item $E^* \in \Id(\HH),$ $\ol{E}=P_{\ol{\M}, \ol{\N}} \ \hat{+}  \ (\{0\}\times \ol{\St}),$ $\ran \ol{E}$ and $\ran(I-\ol{E})$ are closed;
		\item $\M^{\perp}+\N^{\perp}+\St^{\perp}$ and $\ol{\M+\N} +\ol{\St}$ are closed and $\ol{\M+\N} \cap \ol{\St} =\ol{\M} \cap \ol{\St}+\ol{\N}\cap\ol{\St}.$ 
	\end{enumerate} 
In this case, $\ol{E}\in \Id(\HH)$ and 
\begin{equation} \label{IdEq}
\ol{E}=P_{\ol{\M}+\ol{\N}\cap \ol{\St}, \ol{\N}+\ol{\M}\cap \ol{\St}, \ol{\St}+\ol{\M}\cap \ol{\N}}.
\end{equation}
\end{theorem}

\begin{proof} Assume that item $i)$ holds.  Then $\ol{E}$ is super-idempotent and, by Lemma \ref{Tclosed},  $\M^{\perp}+\N^{\perp}+\St^{\perp}$ is closed. Since $E^* \in \Id(\HH)$ then $\ol{E}=(E^*)^*$ is sub-idempotent and hence $\ol{E} \in \Id(\HH).$ Therefore $\ol{E}=P_{\ker(I-\ol{E}),\ker\ol{E},\mul \ol{E}}$ and applying Lemma \ref{propmul}, formula \eqref{IdEq} follows.

By Lemma \ref{propmul}, $\ran \ol{E}=\ol{\M}+\ol{\St}$ and $\ran(I-\ol{E})=\ol{\N}+\ol{\St}.$ Then  $\ol{\M}+\ol{\St}$ and $\ol{\N}+\ol{\St}$ are closed and, by Propositions \ref{estrella} and \ref{Ecerrado} applied to $\ol{E},$ $\ol{\N+\St}+\ol{\M+\St}+\ol{\M+\N}=\ol{\N}+\ol{\St}+\ol{\M}+\ol{\M+\N}=\ol{\St}+\ol{\M+\N}$ is closed.
On the other hand, by Corollary \ref{corclausura},	$$\ol{E}=P_{\ol{\M+\St}\cap \ol{\M+\N}, \ \ol{\N+\St}\cap \ol{\M+\N}, \ \ol{\N + \St}\cap \ol{\M+\St}}.$$  Then 
\begin{align*}
\ol{\M+\N}\cap \ol{\St} &= \ol{\M+\N} \cap \ol{\M+\St} \cap  \ol{\N+\St} \cap \ol{\St}=\ker(I-\ol{E})\cap \ker \ol{E} \cap \ol{\St}\\
&=(\ol{\M}+\ol{\N}\cap \ol{\St}) \cap (\ol{\N}+\ol{\M}\cap \ol{\St}) \cap \ol{\St}\\
&=\ol{\M} \cap \ol{\St}+\ol{\N}\cap\ol{\St}.
\end{align*}

Conversely, assume that item $ii)$ holds. By Lemma \ref{Tclosed}, $$\ol{E}=P_{\ol{\M}, \ol{\N}} \ \hat{+}  \ (\{0\}\times \ol{\St}).$$ Since $\ol{\M+\N} \cap \ol{\St}= \ol{\M} \cap \ol{\St}+\ol{\N}\cap\ol{\St},$ it follows that $ \ol{\M+\N} \cap \ol{\St}=(\ol{\M}+\ol{\N}) \cap \ol{\St}.$
From this fact it can be seen that 
$$\ker(I-\ol{E})=\ol{\M}+\ol{\N}\cap \ol{\St}, \ \ker \ol{E}=\ol{\N}+\ol{\M}\cap \ol{\St} \mbox{ and } \mul \ol{E}=\ol{\St}+\ol{\M}\cap \ol{\N}$$ satisfy the IC. Then $\ol{E} \in \Id(\HH).$ 

Since $\mul \ol{E}$ is closed, it follows that
\begin{align*}
\ol{\ker(I-\ol{E})+\ker \ol{E}} + \ol{\mul \ol{E}}&=\ol{\M+\N}+\ol{\St}+\ol{\M}\cap \ol{\N}\\
&=\ol{\M+\N}+\ol{\St}
\end{align*}
is closed. In a similar fashion, it can be seen that
\begin{align*}
\ol{\ker(I-\ol{E})+\ker \ol{E}} \cap  \ol{\mul \ol{E}} = \ker(I-\ol{E}) \cap \ker \ol{E}.
\end{align*}
 Then, by Proposition \ref{closedstar2} applied to $\ol{E},$ $E^* \in \Id(\HH)$ and $\ran \ol{E},$ $\ran(I-\ol{E})$ are closed.
\end{proof}

\begin{corollary} Let $E=P_{\M,\N,\St}$ such that $ \ol{\M+\N} \cap \ol{\St}= \ol{\M}\cap\ol{\N},$ $\ol{\M+\N} +\ol{\St}$ and $\M^{\perp}+\St^{\perp}+\N^{\perp}$ are closed.
Then:
\begin{enumerate}
\item[1.]  $E^* \in \Id(\HH);$ 
\item[2.] $\ol{E}=P_{\ol{\M},\ol{\N},\ol{\St}};$
\item[3.]  $\ran \ol{E}$ and $\ran(I-\ol{E})$ are closed.
\end{enumerate}
\end{corollary}

\begin{proof} Since $\ol{\M+\N}\cap \ol{\St}=\ol{\M}\cap \ol{\N}=\ol{\M}\cap \ol{\N}\cap \ol{\St}$ and also $(\ol{\M}\cap \ol{\St}+\ol{\N}\cap \ol{\St}) \subseteq \ol{\M+\N}\cap \ol{\St}=\ol{\M}\cap \ol{\N},$ then $\ol{\M}\cap \ol{\N}=\ol{\N}\cap \ol{\St}=\ol{\M} \cap \ol{\St}.$ We then apply Theorem \ref{thmclausura} to get that $E^* \in \Id(\HH),$  $\ol{E}=P_{\ol{\M},\ol{\N},\ol{\St}+\ol{\M}\cap\ol{\N}}=P_{\ol{\M},\ol{\N},\ol{\St}}$ and  $\ran \ol{E}$ and $\ran(I-\ol{E})$ are closed.	
\end{proof}

\subsection*{Acknowledgements}
M. L. Arias was supported in part by UBACyT (20020190100330BA) and FonCyT (PICT 2017-0883).
A. Maestripieri was supported in part by the Interdisciplinary Center for Applied Mathematics at Virginia Tech.
M. Contino, A. Maestripieri  and S. Marcantognini were
supported by CONICET PIP 11220200102127CO.

\end{document}